\documentclass{article} 
\usepackage{mathmag}
\usepackage{amsmath,amsthm}     
\usepackage{graphicx}     
\usepackage{hyperref} 
\usepackage{url}
\usepackage{amsfonts} 
\newtheorem{proposition}{Proposition}

\begin{document}

\title{Hanging cables and spider threads}
\author{Christoph B\"orgers\\               
\scriptsize Department of Mathematics, Tufts University\\    
177 College Avenue, Medford, MA 02155\\          
cborgers@tufts.edu}                     

\maketitle

\begin{abstract}
 It has been known for more than 300 years that 
the shape of an inelastic hanging cable, chain, or rope of uniform linear mass density is the graph of the hyperbolic cosine, up to 
scaling and shifting coordinates. But given two points at which the ends of the cable are attached, {\em how} exactly 
should we scale and shift the coordinates? Many otherwise excellent expositions of the problem are a little vague about that. 
They might for instance give the answer in terms of the tension at the lowest point, but without explaining
how to compute that tension. Here we discuss how to obtain all necessary parameters. To obtain the tension at the lowest point, one has to solve a nonlinear equation numerically. When the two ends of the cable are attached at different heights, a second nonlinear equation must be solved to determine the location of the lowest point. 
When the cable is elastic, think of a thread in a spider's web for instance, the two equations can no longer be decoupled, but they can be solved  using two-dimensional Newton iteration. 
\end{abstract}

\section{Introduction}

\begin{figure}[h]
\centering
\includegraphics[scale=0.47]{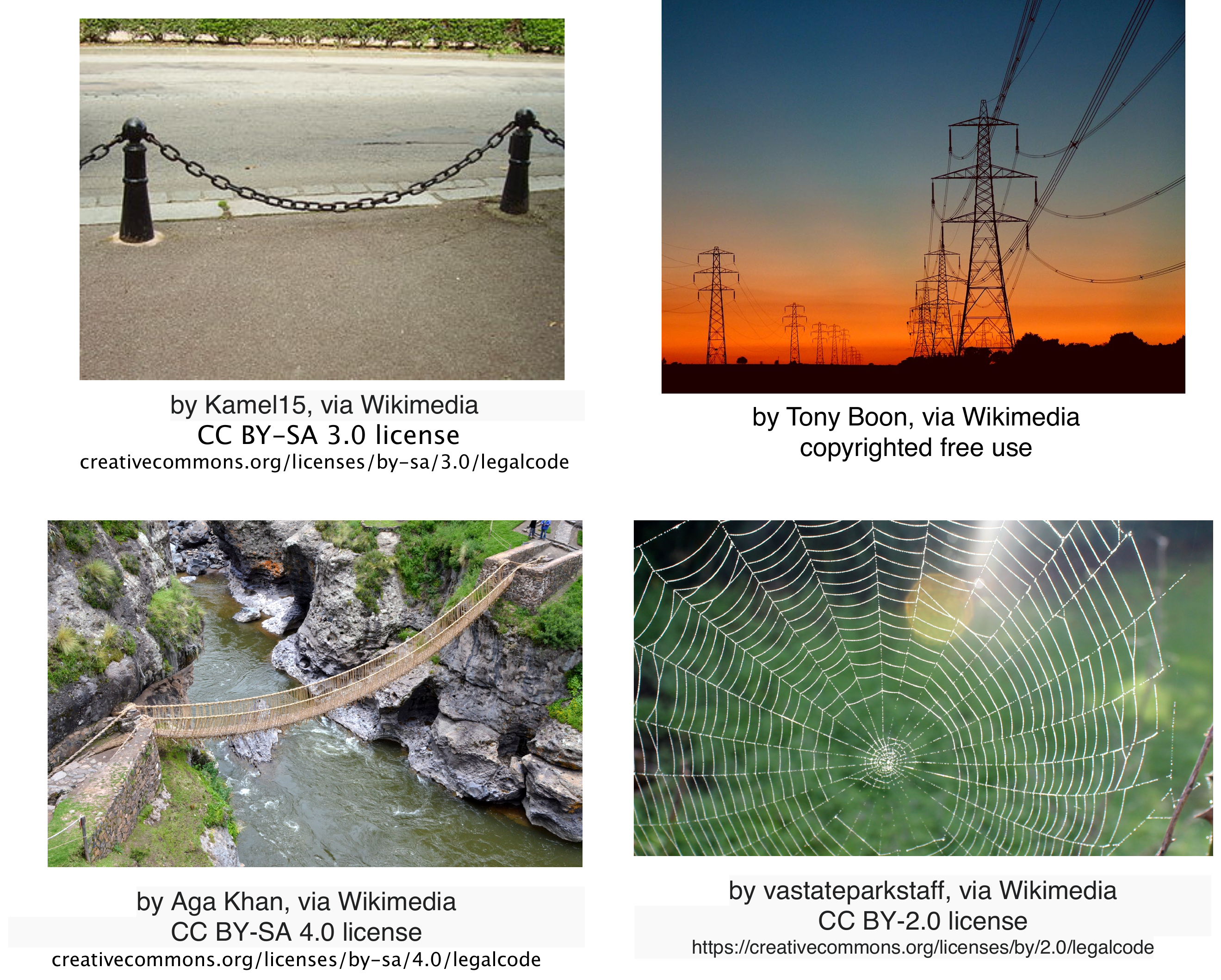}
\caption{Examples of catenaries.}
\label{fig:HANGING_CHAINS}
\end{figure}

The shape of a hanging cable (or chain or rope) is called a {\em catenary}; see Fig.\ \ref{fig:HANGING_CHAINS} for 
examples.
In 1691, in three papers published back-to-back in the
same journal \cite{Bernoulli_catenary,Huygens_catenary,Leibniz_catenary}, 
the {\em inelastic} catenary was found to be described, up to shifting coordinates, by an equation of
the form
\begin{equation}
\label{eq:catenary_equation}
\frac{y}{\lambda} = \cosh \frac{x}{\lambda}.
\end{equation}

\begin{center}
\includegraphics[scale=0.3]{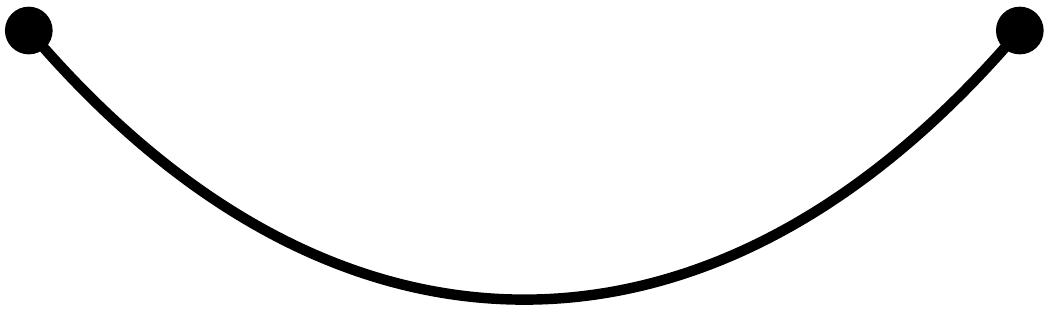}
\end{center}
\noindent
 (The notation was different back then; the notion of hyperbolic cosine did not
exist yet.) The parameter $\lambda>0$ is a length, and we will  refer to it as the {\em shape parameter.} It is also sometimes
called the {\em catenary parameter}. Note that $x$ and $y$ must be scaled the same way. No hanging chain is described by 
$y= \cosh(2x)$, if the same length unit is used for $x$ and for $y$.

Many excellent presentations of the derivation of (\ref{eq:catenary_equation}) are available. I'll give
my own below. To find $\lambda$, one must (numerically) solve a nonlinear equation.
This can be done using Newton's method, and
with  a suitably chosen initial guess, convergence is guaranteed for convexity reasons.

Countless variations have been studied. Perhaps the simplest is the question of 
what happens when 
the two ends are not anchored at the same height. For an example, see the left lower panel of Fig.\ \ref{fig:HANGING_CHAINS}, which depicts the
Queshuachaca  Rope Bridge in Peru. 

\begin{center}
\includegraphics[scale=0.3]{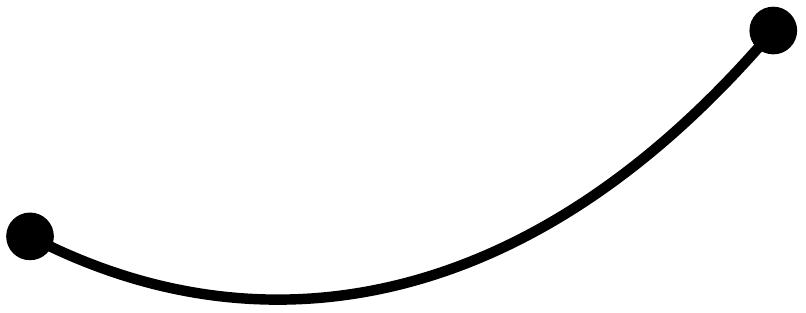}
\end{center}

\noindent
The shape is still a hyperbolic cosine, but 
now the location of the lowest point is no longer obvious by symmetry. Two coupled nonlinear
equations determine the shape parameter and the location of the lowest point. 
There is an algebraic trick by which the system can be 
decoupled, making it possible to solve first for the shape parameter, then for the location of the lowest point.
Each requires the solution  of a  (scalar) nonlinear equation. For both equations, convergence of Newton's method 
is guaranteed, if the initial guess is chosen judiciously, again for convexity reasons.

An elastic cable, for instance a spider thread, is not described by 
a hyperbolic cosine, and in fact it is no longer possible to write $y$ as a function of $x$ explicitly at all. However, both $x$ and $y$ can 
still be written explicitly as functions of $s$ $=$ arc length in the absence of tension. This, too,
has been known for centuries \cite{Elastic_catenary}. Again there is a system of two coupled nonlinear 
equations in two unknowns determining the shape parameter and the lowest point, but 
there is no longer an algebraic trick decoupling the equations. One must solve for the shape parameter
and the lowest point simultaneously. Newton's method in two dimensions, starting with the parameter values for
the inelastic case, does this reliably and with great efficiency.

\section{Inelastic cable with both ends at the same height}

\label{sec:catenary}

We think of a  cable  hanging in an $(x,y)$ plane that is perpendicular to the ground. 
The ends are attached at $(x,y) = (A, H)$ and $(x,y) = (B,H)$, and the length $L$ of the cable
is greater than $B-A$, so the cable sags.
We call the coordinates of the bottom point $x_{\rm min}$ and $y_{\rm min}$. This is the most standard
catenary problem.

\vskip 5pt

\begin{center}
\includegraphics[scale=0.3]{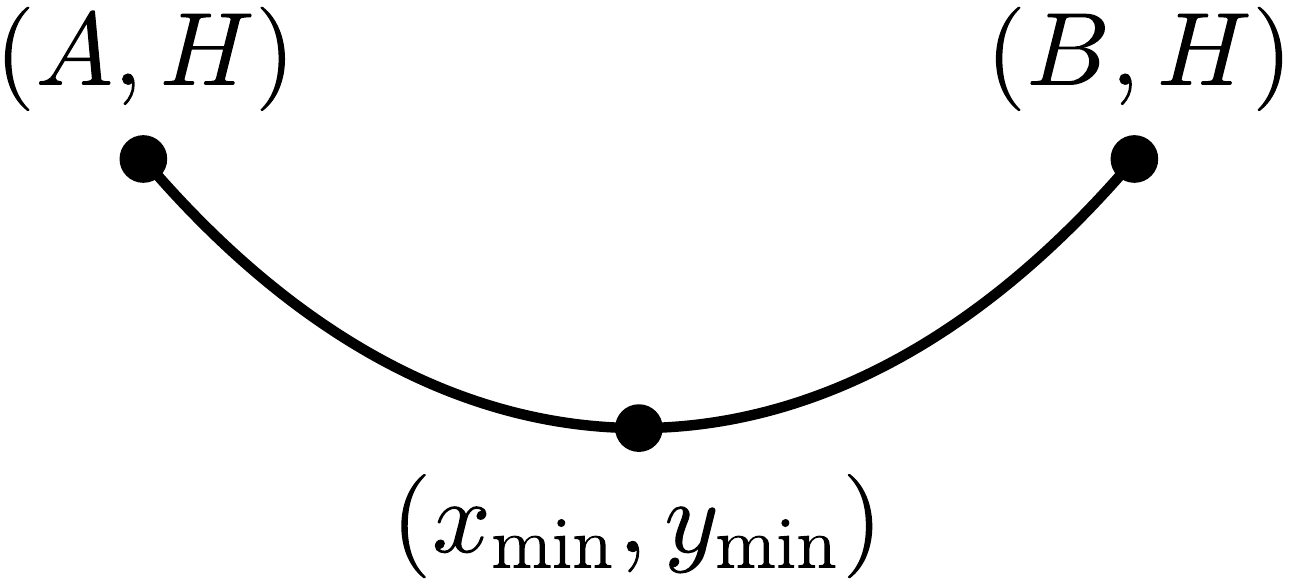}
\end{center}

\subsection{The conventional derivation of the hyperbolic cosine.} 
\label{subsec:catenary_derivation}

Focus on a segment of the cable   between the bottom
point $(x_{\rm min},y_{\rm min})$ and some point  on the right, $(x,y)$ with $x>x_{\rm min}$ and $y>y_{\rm min}$; see
Fig.\ \ref{fig:MY_CHAIN}. We could similarly discuss
a segment between $(x_{\rm min},y_{\rm min})$ and some point on the left, $(x,y)$ with $x<x_{\rm min}$ and $y>y_{\rm min}$, with analogous conclusions.
The part of the cable to the right
of $(x,y)$ pulls on this segment with a certain force tangential to the cable. 
We  call  the magnitude of this 
force the {\em tension} at $(x,y)$, and denote it by $T$.

\begin{figure}[h]
\begin{center}
\includegraphics[scale=.3]{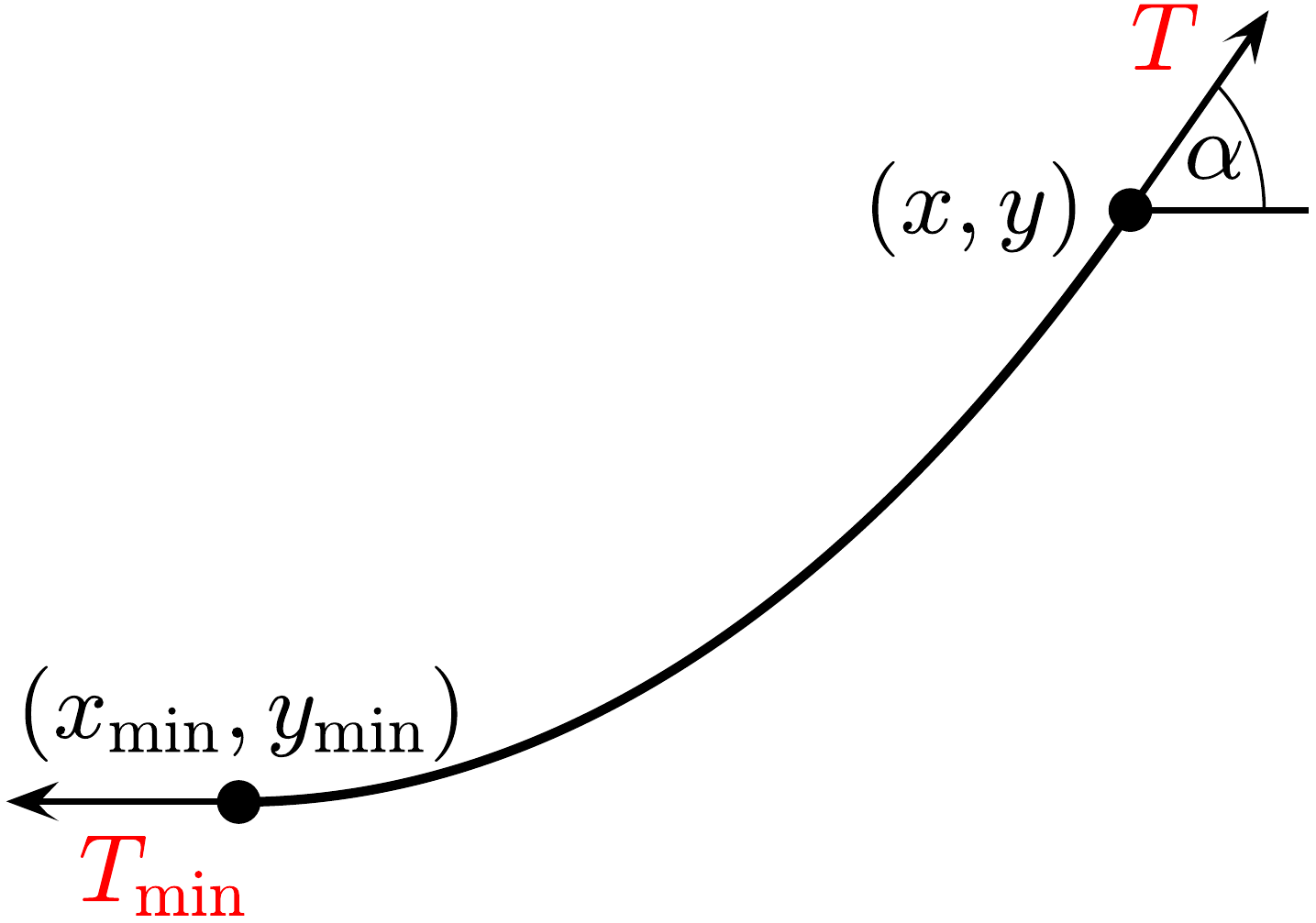}
\caption{Tension forces on a segment of the hanging cable.} 
\label{fig:MY_CHAIN} 
\end{center}
\end{figure}
\noindent

If you are like me and feel a slight discomfort now, perhaps not being {\em entirely} sure that you know what ``tension" really means, and in what sense parts of the cable
pull on other parts of the cable, then read the discussion of the {\em elastic} cable below for a better explanation of the inelastic one.

The part of the cable to the left of $(x_{\rm min},y_{\rm min})$ pulls on the segment with a force of magnitude $T_{\rm min}$, 
the tension at the bottom point $(x_{\rm min},y_{\rm min})$. 
Since the cable is stationary, the horizontal components of the two tension forces must balance: 
\begin{equation}
\label{eq:hori}
T \cos \alpha = T_{\rm min}
\end{equation}
where the definition of $\alpha \in [0, \frac{\pi}{2} )$ is indicated in Fig.\ \ref{fig:MY_CHAIN}.
 The notation $T_{\rm min}$ is doubly appropriate; not only is it the tension at the 
lowest point, it is also the minimal tension, as eq.\ (\ref{eq:hori}) shows.

Similarly, the weight of the cable segment between $(x_{\rm min},y_{\rm min})$ and $(x,y)$ must balance the vertical tension forces.
Denote by $\Delta s$
the length of the segment between $(x_{\rm min},y_{\rm min})$ and $(x,y)$. Also, denote by $\rho$ the  {\em linear mass
density}, that is the mass per unit length, of the cable. We assume $\rho$ to be constant.  The mass of our segment of
length $\Delta s$ is $\rho \Delta s$. 
Therefore the weight 
of the segment between $(x_{\rm min}, y_{\rm min})$ and $(x,y)$  is $\rho g \Delta s $, where $g$ is the gravitational acceleration.
This weight must be balanced by the vertical component of the tension force at $(x,y)$, since at $(x_{\rm min}, y_{\rm min})$, the vertical
component of the tension force is zero:
\begin{equation}
\label{eq:verti}
T \sin \alpha = \rho g \Delta s.
\end{equation}

We divide  (\ref{eq:verti}) by (\ref{eq:hori}) to obtain
$$
 \tan \alpha =  \frac{\rho g \Delta s}{T_{\rm min}}. 
$$
If we think of $y$ as a function of $x$ in Fig.\ \ref{fig:MY_CHAIN}, then $\tan \alpha$ is the derivative of $y$ with respect to $x$. We'll denote this derivative by 
$y'(x)$. So 
\begin{equation}
\label{eq:y_prime}
y'(x)  = \frac{\rho g \Delta s}{T_{\rm min}}.
\end{equation}
The arc length $\Delta s$ is a function of $x$:
$$
\Delta s = \int_{x_{\rm min}}^x \sqrt{1+y'(u)^2} ~\! du, 
$$
where again $y'$ denotes the derivative of $y$,  and we use the letter $u$ for 
no better reason than that it isn't $x$, $y$, or $s$ (which we reserve for arc length). 
Therefore
$$
y'(x) = \frac{\rho g}{T_{\rm min}} \int_{x_{\rm min}}^x \sqrt{1+y'(u)^2} ~\! du.
$$
Differentiating both sides, we get the second-order differential equation

\begin{equation}
\label{eq:diffeq}
y''(x) = \frac{\rho g}{T_{\rm min}} \sqrt{1+y'(x)^2}.
\end{equation}

To simplify the notation, we write 
\begin{equation}
\label{defl}
\lambda = \frac{T_{\rm min}}{\rho g}. 
\end{equation}
Note that $\lambda$ is a length, since $T_{\rm min}$ is  a force, and $\rho g$ is a force per unit length. 
This is the parameter that we will call the {\em shape parameter}.
Equation (\ref{eq:diffeq}) implies that $z(x) = y'(x)$ satisfies the first-order differential equation
$$
z'(x) =  \frac{1}{\lambda} \sqrt{1+z(x)^2}.
$$
By separation of variables, remembering that $\int \frac{1}{\sqrt{1+z^2}} ~\! dz = \sinh^{-1} z + C$,  and using that $z(x_{\rm min}) = y'(x_{\rm min})=0$, 
we find
$$
z(x) = y'(x)  = \sinh \left(   \frac{x-x_{\rm min}}{\lambda} \right). 
$$
We integrate one more time to obtain 
\begin{equation}
\label{y_gleich}
y(x) = \lambda \cosh \left(  \frac{x-x_{\rm min}}{\lambda} \right)  - \lambda + y_{\rm min}.
\end{equation}
We picked the constant of integration so that $y$ comes out to be $y_{\rm min}$ when $x = x_{\rm min}$. 
We re-write eq.\ (\ref{y_gleich}) as
\begin{equation}
\label{eq:y_is}
\frac{y-y_{\rm min}}{\lambda} = \cosh \left(  \frac{x- x_{\rm min}}{\lambda}  \right) - 1.
\end{equation}

The most remarkable thing about the catenary has now been said: With appropriate shifting and scaling 
of the coordinates
(with $x$ and $y$ scaled exactly the same way --- that is, using the same length units for $x$ and $y$), it is a hyperbolic
cosine. But what are $x_{\rm min}$, $y_{\rm min}$, and $\lambda$?

By symmetry, 
\begin{equation}
\label{eq:x_min} 
x_{\rm min} = \frac{A+B}{2}.
\end{equation}
If we knew $\lambda$ as well, then $y_{\min}$ could
be obtained from (\ref{y_gleich}), using that $y= H$ when $x=B$:
\begin{equation}
\label{eq:y_min} 
y_{\rm min} = H - \lambda \cosh \left(  \frac{B-A}{2 \lambda} \right)  + \lambda.
\end{equation}
However, we still have to determine the shape parameter $\lambda$, or equivalently (see eq.\ (\ref{defl})) the tension $T_{\rm min}$ at the lowest
point.

\subsection{The equation for $\lambda$, or equivalently, for the tension at the lowest point.} 

We obtain $\lambda$ from the fact that the cable has length $L$. By eq.\ (\ref{y_gleich}), this means:
\begin{eqnarray}
&~&~~
\nonumber
\int_A^B \sqrt{1+\sinh^2 \left( \frac{x-x_{\rm min}}{\lambda} \right)}~ dx =  L\\
\nonumber
&\Leftrightarrow&~~ \int_A^B \cosh \left( \frac{x-x_{\rm min}}{\lambda} \right) ~ dx = L \\
\label{eq:transcendental_0}
&\Leftrightarrow&~~ \left.  \lambda \sinh \left( \frac{x-x_{\rm min}}{\lambda} \right) \right|_A^B  = L \\
\label{eq:transcendental}
&\Leftrightarrow& ~~ \sinh \left( \frac{B-A}{2 \lambda}  \right) =  \frac{L}{2 \lambda}.
\end{eqnarray} 
Equation (\ref{eq:transcendental}) is the nonlinear equation that determines $\lambda$.
It is convenient here to make a minor change of coordinates:
\begin{equation}
\label{eq:xi_of_lambda}
\xi = \frac{B-A}{2 \lambda}, 
\end{equation}
so eq.\ 
(\ref{eq:transcendental}) becomes
\begin{equation}
\label{eq:xi_is}
\sinh \xi - \frac{L}{B-A} \xi = 0.
\end{equation}

\subsection{Finding $\lambda$.} To find $\lambda$, we solve eq.\ (\ref{eq:xi_is}) for $\xi$. The following proposition provides details.

\begin{proposition} 
\label{proposition:unique_solution}
Equation (\ref{eq:xi_is}) has exactly one positive solution $\xi$, and consequently
eq.\ (\ref{eq:transcendental}) has exactly one positive solution $\lambda$. Newton's method, applied
to (\ref{eq:xi_is})  with initial guess
$\sqrt{6 \frac{L}{B-A}}$, is  assured to converge to the positive solution.
\end{proposition}

\begin{proof} 
Existence and uniqueness of a positive solution follow from the convexity of 
$\sinh \xi$ for $\xi \geq 0$, and from $\sinh(0)=0$, $\sinh'(0)=1$, and
$\frac{L}{B-A}>1$ (which holds by assumption --- the cable sags);  see Fig.\ \ref{fig:NEWTON}A.

To show that Newton's method, when starting at $\sqrt{6 \frac{L}{B-A}}$, converges to the positive solution,
it suffices, because of the convexity of the graph of $\sinh \xi - \frac{L}{B-A} \xi$, to prove that $\sqrt{6 \frac{L}{B-A}}$ is an upper bound for the positive
 solution; see Fig.\  \ref{fig:NEWTON}B. 

The following argument proves that $\sqrt{6\frac{L}{B-A}}$ is indeed an upper bound for the positive solution of eq.\ (\ref{eq:xi_is}). 
Since $\sinh \xi = \xi + \frac{\xi^3}{3!} + \frac{\xi^5}{5!} + \ldots$ we have for $\xi>0$: 
$$
\sinh \xi > \frac{L}{B-A} \xi  ~~\Leftarrow~~ \frac{\xi^3}{6} \geq \frac{L}{B-A} \xi ~~\Leftrightarrow~~ \frac{\xi^2}{6} \geq \frac{L}{B-A}   ~~ \Leftrightarrow~~ 
\xi \geq  \sqrt{6 \frac{L}{B-A}}.
$$
So if $\xi \geq  \sqrt{6 \frac{L}{B-A}}$, then $\xi$ does not solve eq.\ (\ref{eq:xi_is}). In other words, any solution of 
eq.\ (\ref{eq:xi_is}) is smaller than $\sqrt{6 \frac{L}{B-A}}$.
\end{proof}

\begin{figure}[h]
\begin{center}
\includegraphics[scale=0.5]{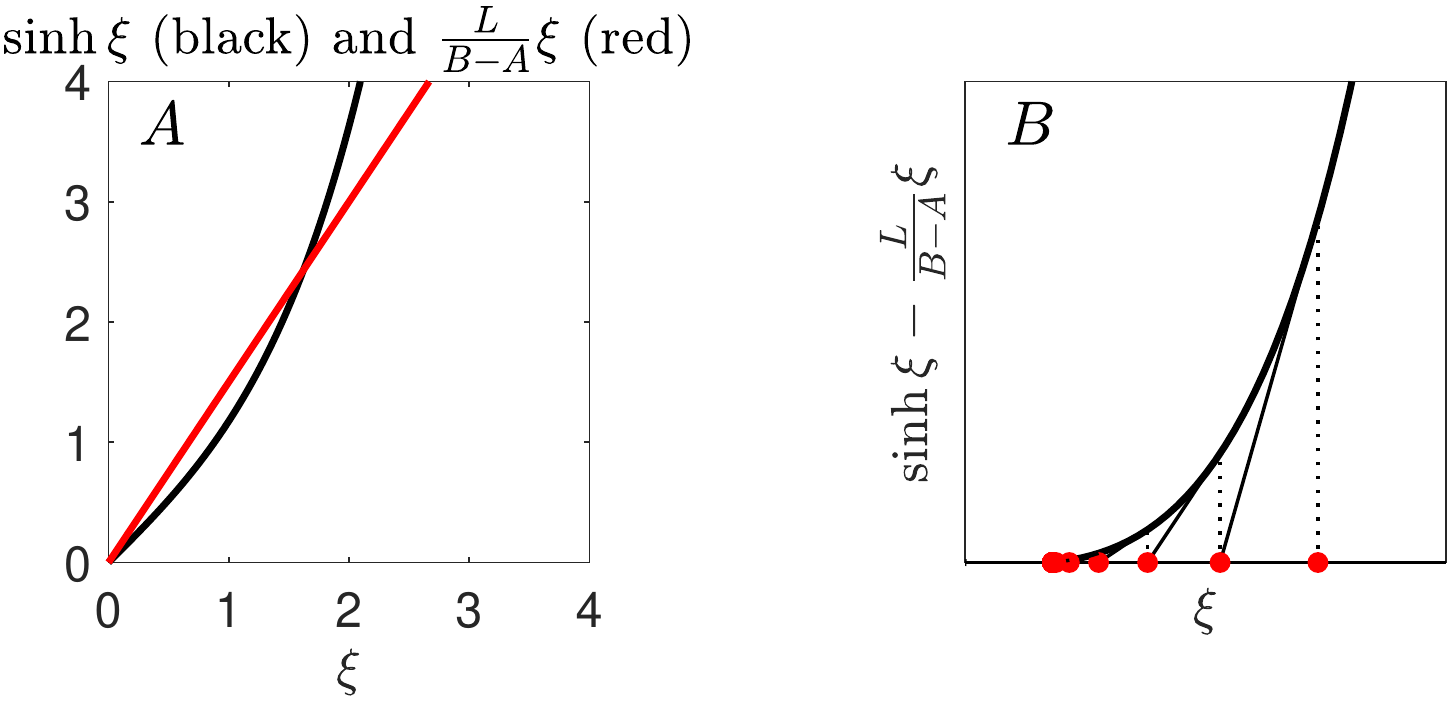}
\caption{A: The reason why $\sinh \xi = \frac{L}{B-A} \xi$ has a unique positive solution when $\frac{L}{B-A} > 1$. B: Illustration of Newton's method applied to $\sinh \xi - \frac{L}{B-A} \xi = 0$, starting to the right of the positive solution.}
\label{fig:NEWTON}
\end{center}
\end{figure}

\subsection{Summary.} 
The shape of the inelastic cable, hung up so that both ends are at the same height, is found as follows.

\begin{enumerate}
\item Find the positive solution of eq.\ (\ref{eq:xi_is}) using Newton's method with 
 initial guess $\sqrt{6 \frac{L}{B-A}}$, 
and define 
$
\lambda = \frac{B-A}{2 \xi}.
$
\item Define $x_{\rm min}$ and $y_{\rm min}$ according to eqs.\ (\ref{eq:x_min}) and (\ref{eq:y_min}). 
\item The shape of the hanging cable is given by eq.\ (\ref{eq:y_is}). 
\end{enumerate}

\subsection{Afterthoughts.}
\vskip 5pt

\begin{enumerate}
\item 
The equation is derived by considering the balance of vertical and
horizontal forces on a segment between the lowest point 
and another point. However, this implies balance of vertical and horizontal forces on any segment 
along the cable.

\item 
The shape parameter is computed without knowledge of $\rho$. The weight of the cable 
is irrelevant to its shape. 

\item 
The {\em tension} in the cable, of course, does depend on the weight. The tension at the lowest point,
for instance, is $T_{\rm min} = \rho g \lambda$ (compare eq.\ (\ref{defl})).

\item 
As $L$  tends to $B-A$, the positive solution $\xi$ of (\ref{eq:xi_is}) tends to $0$. Therefore $\lambda$ tends to $\infty$, 
and so does $T_{\rm min} = \rho g \lambda$. Therefore it is impossible for the cable not to sag at all; 
that would require infinite tension.
\end{enumerate}

\subsection{Inverted catenaries in architecture.}

There are countless examples of arches approximately in the shape of  (upside-down) catenaries in architecture, as well as domes approximately in the shape of {\em catenary rotation surfaces} \cite[Chapter 7]{Gohnert}.  Such a surface is obtained
by rotating a catenary around its  vertical axis of symmetry \cite{Lopez_2022}.\footnote{It is not to be confused with the {\em catenoid}   obtained 
by rotating a catenary around the {\em horizontal} ($x$-)axis; see \cite{de_la_Grandville_2022} for a recent fascinating discussion of catenoids.}

Catenary arches  have a special stability
property, the mirror image of the force balance that leads to the equation of the catenary; the horizontal 
and vertical forces on any segment are in balance. A catenary rotation surface does not have the analogous property but is not 
far from a surface that does \cite{Boehme_et_al_1980, Lopez_2022}. 

Figure \ref{fig:CATENARIES_IN_ARCHITECTURE}
shows examples of catenary arches and domes in panels A--C. Panel D of the figure shows the Gateway Arch in St.\ Louis, and it is {\em not} an inverted catenary; it is instead a curve of the form
$$
\frac{y}{\lambda_y} = - \cosh \frac{x}{\lambda_x}
$$
(after shifting the coordinates appropriately), with $\lambda_x \approx 1.45 \lambda_y$ \cite{Osserman}. This is called a {\em weighted catenary}. 

\begin{figure}[h] 
\begin{center}
\includegraphics[scale=0.35]{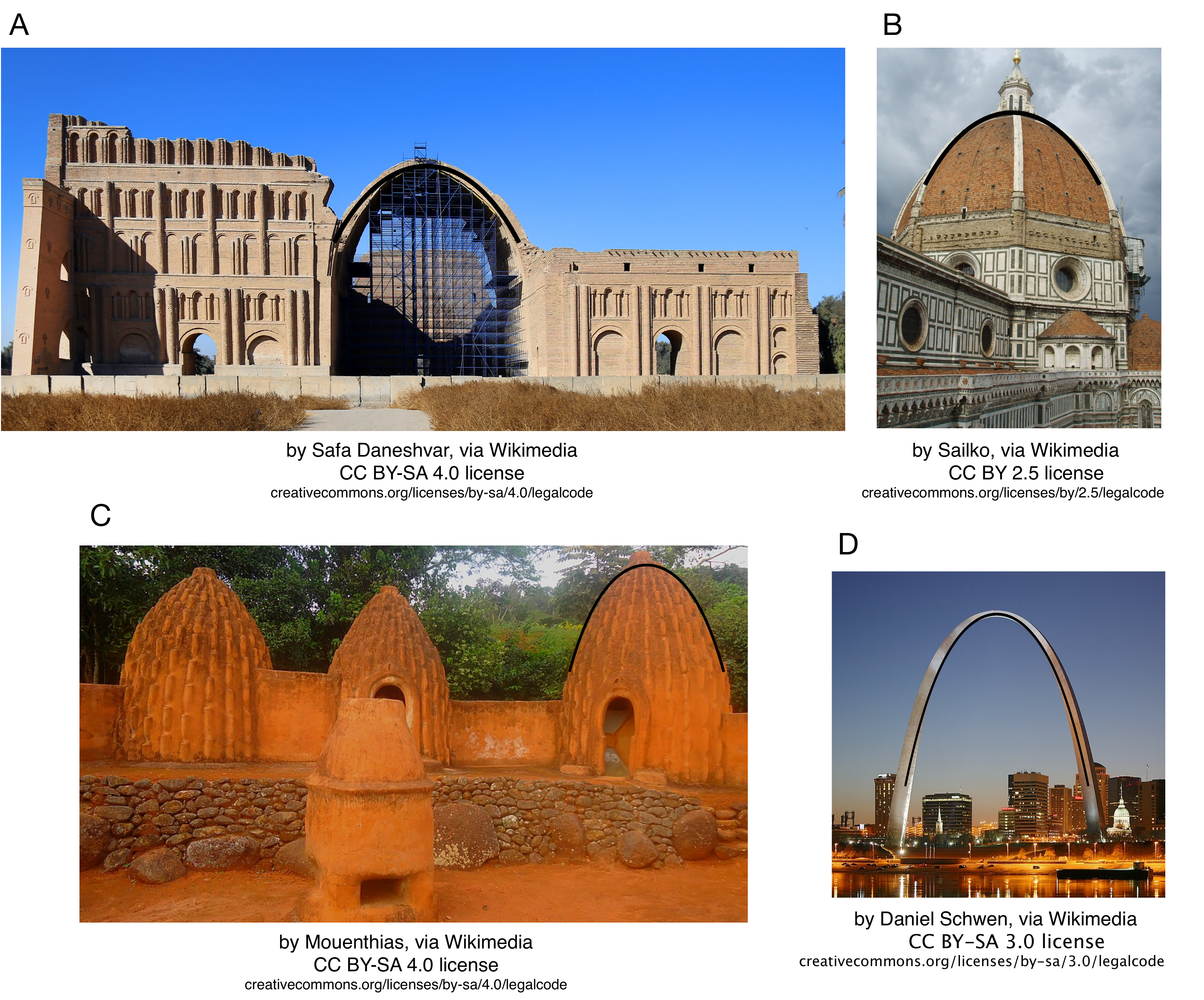}
\caption{Examples of catenaries and catenary domes in architecture. A: The Arch of Ctesiphon, a Persian monument in present-day Iraq, about 1500 years old. B: The dome of the cathedral of Florence, built between 1296 and 1436. C: 
Traditional houses of the Musgum people in Cameroon. D: The Gateway Arch in St.\ Louis. It is a {\em weighted} catenary.
The black curves are catenaries in panels A--C, and a weighted catenary in panel D.}
\label{fig:CATENARIES_IN_ARCHITECTURE}
\end{center}
\end{figure}

\section{Inelastic cable with ends at different heights}
\label{sec:different_heights}

Now we assume that 
the ends are attached at $(x,y) = (A, H)$ and $(x,y) = (B,K)$, and without loss of generality
$
H \leq  K.
$
We assume that  the length $L$ of the cable is greater than the distance between $(A,H)$ and $(B,K)$, so the cable
sags:
\begin{equation}
\label{eq:long_enough}
L > \sqrt{(B-A)^2+(K-H)^2}.
\end{equation}

\subsection{It's still a hyperbolic cosine.}

Our previous arguments still show that 
the solution is of the form given by eq.\ (\ref{eq:y_is}), repeated here for convenience:
$$
\hskip 108pt
\frac{y-y_{\rm min}}{\lambda} = \cosh \left(  \frac{x- x_{\rm min}}{\lambda}  \right) - 1.
\hskip 84pt
(\ref{eq:y_is}) 
$$

\begin{center}
\includegraphics[scale=0.3]{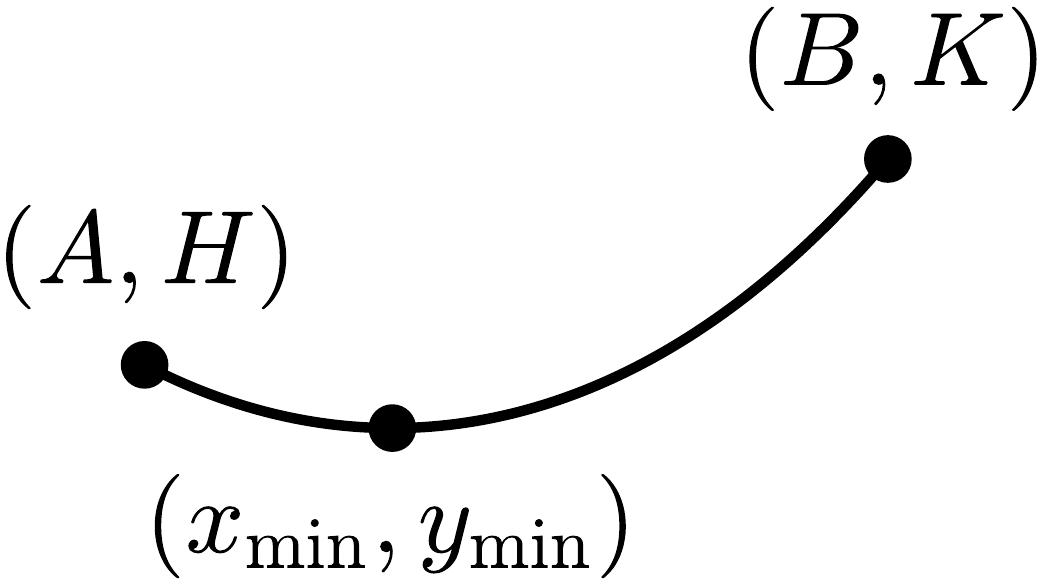}
\end{center}

The complication is that there is no symmetry argument telling us the value of $x_{\rm min}$ any 
longer. In fact, $x_{\rm min}$ could even be to the left of $A$.

\vskip 5pt
\begin{center}
\includegraphics[scale=0.3]{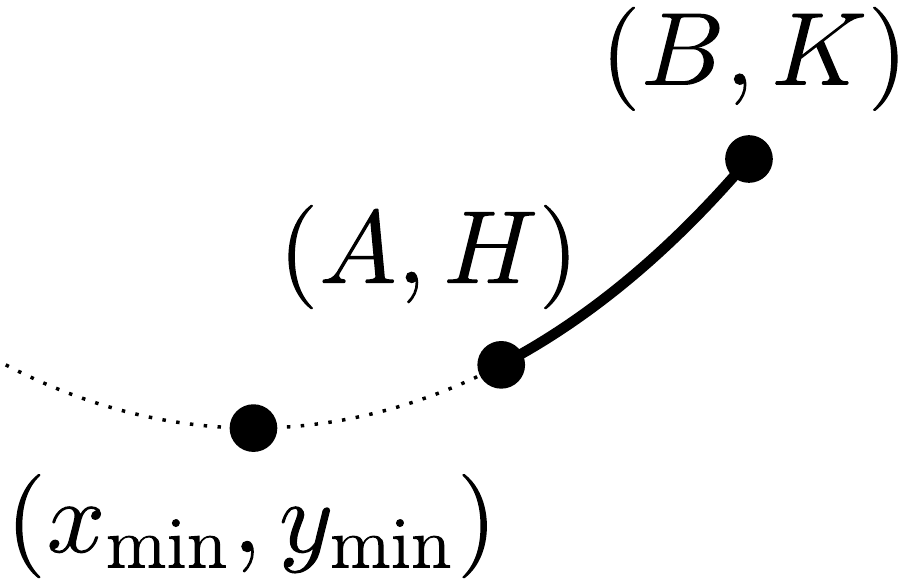}
\end{center}

\subsection{Two equations for the two unknowns $\lambda$ and $x_{\rm min}$.}

The three parameters $\lambda$, $x_{\rm min}$, and $y_{\rm min}$ must be chosen so that three conditions hold:
$$
y(A) = H, ~~~ y(B) = K, ~~~ \mbox{length of cable} = L.
$$
However, we can easily derive two equations for the two parameters $\lambda$ and $x_{\rm min}$: 
$$
y(B) - y(A) = K-H, ~~~ \mbox{length of cable} = L, 
$$
or explicitly, using 
(\ref{y_gleich}) and (\ref{eq:transcendental_0}), 
\begin{eqnarray}
\label{eq:cosh}
 \cosh \left( \frac{B-x_{\rm min}}{\lambda} \right) &-& \cosh \left( \frac{A-x_{\rm min}}{\lambda} \right) ~~=~
\frac{K-H}{\lambda}, \\
\label{eq:sinh}
\sinh \left( \frac{B- x_{\rm xmin}}{\lambda} \right) &-& \sinh \left( \frac{A- x_{\rm xmin}}{\lambda} \right)  ~=~ \frac{L}{\lambda}.
\end{eqnarray}
Once $\lambda$ and $x_{\rm min}$ are known, $y_{\rm min}$  can be obtained from $y(B)=K$ using (\ref{y_gleich}):
\begin{equation}
\label{eq:y_min_general}
y_{\rm min} = K - \lambda \cosh \left( \frac{B-x_{\rm min}}{\lambda} \right) + \lambda.
\end{equation}

\subsection{Finding $\lambda$.} Now there is an algebraic trick. We square (\ref{eq:cosh}) and (\ref{eq:sinh}), subtract them from each other, and use 
$\cosh^2 u - \sinh^2 u = 1$ and $\cosh u \cosh v - \sinh u \sinh v = \cosh(u-v)$ for all $u$ and $v$. We thereby get this:
$$
2 - 2 \cosh \left( \frac{B-A}{\lambda} \right) = \frac{(K-H)^2-L^2}{\lambda^2}, 
$$
or equivalently, 
\begin{equation}
\label{eq:first_equation_for_lambda}
\cosh \left( \frac{B-A}{\lambda} \right) - 1 = \frac{L^2 - (K-H)^2}{2 \lambda^2}.
\end{equation}
This is a single equation for $\lambda$. It can be written in a more appealing way by using one more hyperbolic trigonometric formula: $\cosh u = 1+ 2 \sinh^2 \frac{u}{2}$. With that (\ref{eq:first_equation_for_lambda}) becomes
$$
\sinh^2 \left( \frac{B-A}{2 \lambda} \right)  = \frac{L^2-(K-H)^2}{4 \lambda^2}, 
$$
or equivalently, 
\begin{equation}
\label{eq:2nd_equation_for_lambda}
\sinh \left( \frac{B-A}{2 \lambda} \right) =  \frac{\sqrt{L^2-(K-H)^2}}{2 \lambda}.
\end{equation}

Equation (\ref{eq:2nd_equation_for_lambda}) is precisely the same as eq.\ (\ref{eq:transcendental}), except that $L$ has been replaced by $\sqrt{L^2-(K-H)^2}$. 
Notice also that (\ref{eq:long_enough}) implies that $\sqrt{L^2-(K-H)^2} > B-A$.  Proposition \ref{proposition:unique_solution} therefore applies, with $L$ replaced by $\sqrt{L^2-(K-H)^2}$. Equation (\ref{eq:2nd_equation_for_lambda}) has a unique positive solution $\lambda$, and we can compute $\xi = \frac{B-A}{2 \lambda}$ using Newton's 
method, applied to 
\begin{equation}
\label{eq:Newton_eq}
\sinh \xi - \frac{\sqrt{L^2-(K-H)^2}}{B-A} ~ \xi = 0,
\end{equation}
with initial guess $\sqrt{\frac{6 \sqrt{L^2-(K-H)^2}}{B-A}}$.

\subsection{Finding $x_{\rm min}$.} Once $\lambda$ is known, $x_{\rm min}$ can be obtained by solving eq.\ (\ref{eq:cosh}) for $x_{\rm min}$.
The following proposition provides details.

\begin{proposition} Let $A<B$, $H \leq K$, and $\lambda>0$. Let, for $x \in \mathbb{R}$, 
\begin{equation}
\label{eq:defg}
g(x) = \cosh \left( \frac{B-x}{\lambda} \right) - \cosh  \left( \frac{A-x}{\lambda} \right) - \frac{K-H}{\lambda}
\end{equation}
so  eq.\ (\ref{eq:cosh})  becomes $g(x_{\rm min}) = 0$. 
\begin{enumerate}
\item[(a)] $g$ is strictly decreasing with $\lim_{x \rightarrow -\infty} = \infty$ and $\lim_{x \rightarrow \infty} g(x) = -\infty$, so there is a unique solution, $x_{\rm min}$, 
of $g(x) = 0$, 
\item[(b)] $g \left( \frac{A+B}{2} \right) \leq 0$, so $x_{\rm min} \leq \frac{A+B}{2}$, 
\item[(c)] $g''(x)>0$ for $-\infty < x < \frac{A+B}{2}$, and 
\item[(d)] Newton's method for $g(x)=0$, starting with the initial guess $\frac{A+B}{2}$, converges to the unique solution, $x_{\rm min}$, of $g(x)=0$.
\end{enumerate}
\end{proposition}

See Fig.\ \ref{fig:PLOT_FICTICIOUS_G} for illustration.

\begin{figure}[h]
\begin{center}
\includegraphics[scale=0.3]{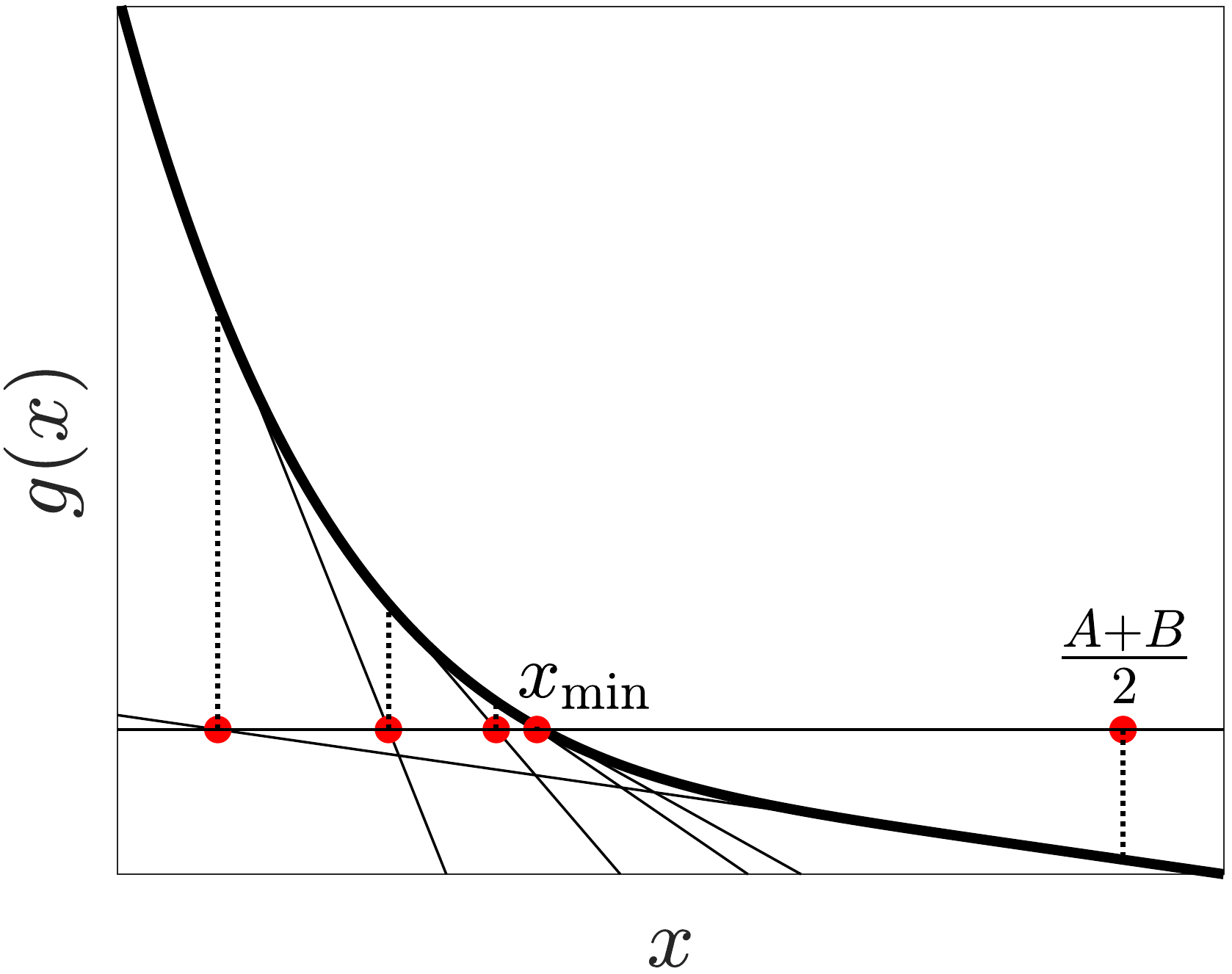}
\caption{Illustration of Newton's method applied to $g(x)=0$.}
\label{fig:PLOT_FICTICIOUS_G}
\end{center}
\end{figure}

\begin{proof}
(a) For all $x$, 
$$
g'(x) = \frac{1}{\lambda} \left( \sinh \left( \frac{A-x}{\lambda} \right) -\sinh \left( \frac{B-x}{\lambda} \right) \right) <0
$$
because $\lambda>0$, $A< B$, and $\sinh$ is a strictly increasing function. Using the definition of $\cosh$, 
\begin{equation}
\label{eq:g_explicit}
g(x) = \frac{e^{ (B-x)/\lambda} + e^{- (B-x)/\lambda}}{2} -\frac{e^{ (A-x)/\lambda} + e^{- (A-x)/\lambda}}{2}  - \frac{K-H}{\lambda}
\end{equation}
As $x \rightarrow -\infty$, (\ref{eq:g_explicit}) equals 
$$
\frac{e^{(B-x)/\lambda} - e^{(A-x)/\lambda}}{2} + O(1)
$$
(the notation $O(1)$ means ``terms that remain bounded in the limit"), and 
$$
\frac{e^{(B-x)/\lambda} - e^{(A-x)/\lambda}}{2} = e^{(A-x)/\lambda}~  \frac{e^{(B-A)/\lambda} - 1}{2} \rightarrow \infty
$$
as $x \rightarrow -\infty$. One sees in a similar way that $g(x) \rightarrow -\infty$ as $x \rightarrow \infty$.
\vskip 5pt
\noindent
(b) 
$$
g \left( \frac{A+B}{2} \right) = - \frac{K-H}{\lambda} \leq 0
$$
because $H \leq K$.

\vskip 5pt
\noindent
(c) 
$$
g''(x) = \frac{1}{\lambda^2} \left( \cosh \left( \frac{B-x}{\lambda} \right) -\cosh \left( \frac{A-x}{\lambda} \right) \right) = \frac{1}{\lambda^2} \left( g(x) + \frac{K-H}{\lambda} \right)
$$
is strictly decreasing since $g$ is known to be strictly decreasing by (a). Since $g'' \left( \frac{A+B}{2} \right) = 0$, (c) follows.

\vskip 5pt
\noindent
(d) After (a)--(c) have been proved, this is so clear pictorially (see Fig.\ \ref{fig:PLOT_FICTICIOUS_G}) that we'll refrain from proving it analytically.
\end{proof}

\subsection{Summary.} The shape of the inelastic cable, hung up so that the two ends are at different heights, is found as follows.

\begin{enumerate}
\item Find the positive solution of eq.\ (\ref{eq:Newton_eq}) using Newton's method with 
 initial guess $\sqrt{6 \frac{\sqrt{L^2-(K-H)^2}}{B-A}}$, 
and define 
$
\lambda = \frac{B-A}{2 \xi}.
$
\item With $g$ defined as in (\ref{eq:defg}), solve $g(x_{\rm min})=0$ for $x_{\rm min}$, using Newton's method with initial guess $\frac{A+B}{2}$.
\item Compute $y_{\rm min}$ from eq.\ (\ref{eq:y_min_general}). 
\item The shape of the hanging cable is given by eq.\ (\ref{eq:y_is}). 
\end{enumerate}

Two examples are shown in Fig.\ \ref{fig:TWO_EXAMPLES_NEW}.

\begin{figure}[h]
\begin{center}
\includegraphics[scale=0.5]{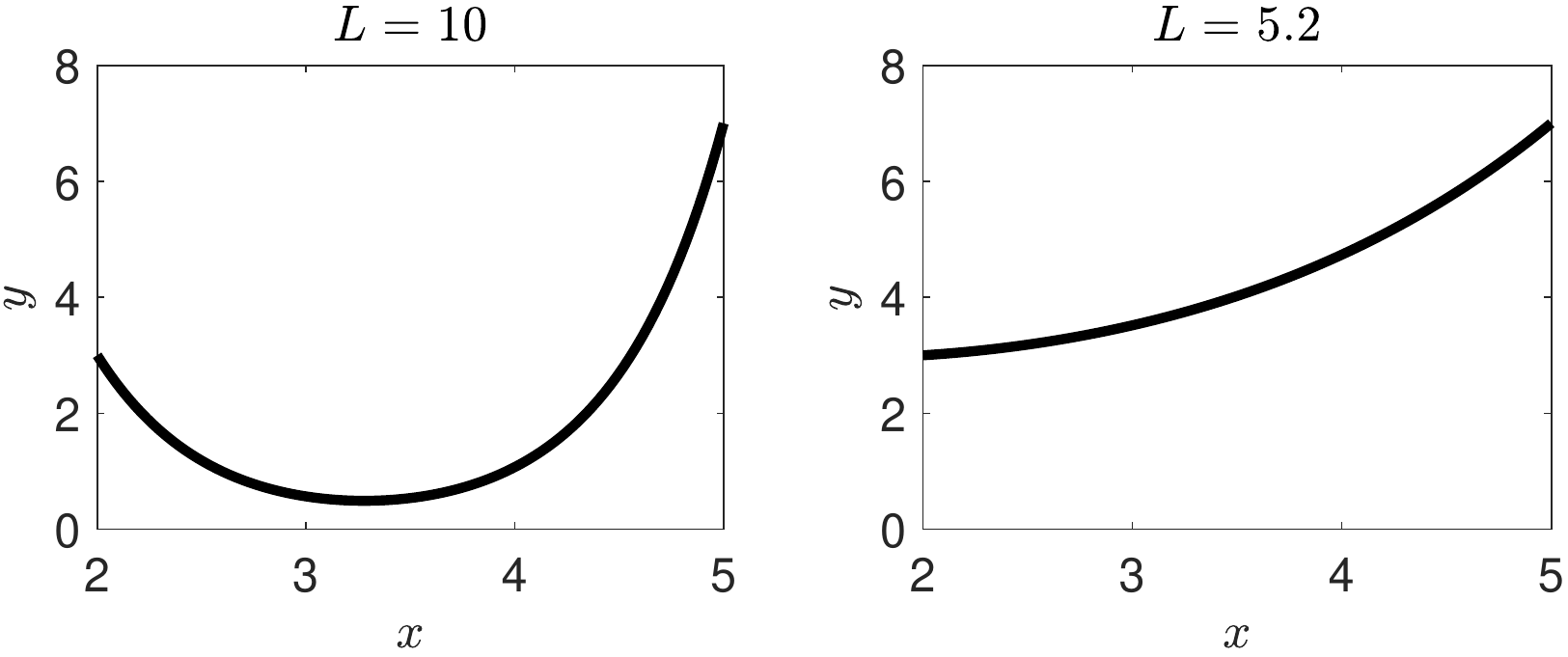}
\caption{A long and a short hanging cable, with $A=2$, $B=5$, $H=3$, $K=7$.}
\label{fig:TWO_EXAMPLES_NEW}
\end{center}
\end{figure}

\subsection{Afterthought.} 
As $L$ tends to $ \sqrt{(A-B)^2+(K-H)^2}$, the solution of  (\ref{eq:Newton_eq}) tends to zero, and
therefore $\lambda$ tends to $\infty$. Again we see that it is impossible for the cable to have 
no sag at all.

\subsection{Arc length parametrization.} We will parametrize the hanging cable with respect to arc length. 
That's entirely unnecessary, but it 
will make the analogy with the elastic case discussed later   more transparent. 
 
From eq.\ (\ref{eq:y_is}), we see that the arc length $s$ between the left end point of the cable and the point
at $x \in [A,B]$ is 
$$
s = \int_A^x \sqrt{1 + \sinh^2 \left( \frac{u-x_{\rm min}}{\lambda} \right)} ~ du.
$$
Using 
$\sqrt{1+\sinh^2} =\cosh$ we evaluate the integral and find
\begin{equation}
\label{x_to_s}
s = 
\lambda \sinh \left( \frac{x-x_{\rm min}}{\lambda} \right) + \lambda \sinh \left( \frac{x_{\rm min}-A}{\lambda} \right).
\end{equation}
The arc length parameter associated with $x_{\rm min}$, in particular, is 
\begin{equation}
\label{s_min}
s_{\rm min} = \lambda \sinh \left( \frac{x_{\rm min}-A}{\lambda} \right).
\end{equation}
Solving (\ref{x_to_s}) for $x$ and using   (\ref{s_min}), we find the relation between $x$ and the arc length $s$:
\begin{equation}
\label{eq:x_of_s}
\frac{x-x_{\rm min}}{\lambda} = \sinh^{-1} \left(  \frac{s-s_{\rm min}}{\lambda} \right).
\end{equation}
With that,  (\ref{eq:y_is}) becomes
\begin{equation}
\label{eq:y_of_s}
\frac{y-y_{\rm min}}{\lambda}  = \cosh \sinh^{-1} \left( \frac{s-s_{\rm min}}{\lambda} \right)   -1 .
\end{equation}
For any $u \in \mathbb{R}$, 
$
\cosh \sinh^{-1}(u) = \sqrt{1+u^2}.
$
(This follows  from $\cosh^2 - \sinh^2 = 1$.) Therefore eq.\ (\ref{eq:y_of_s})  can also be written like this: 
\begin{equation}
\label{eq:y_of_s_rewritten}
\frac{y-y_{\rm min}}{\lambda}  = \sqrt{1+\left( \frac{s-s_{\rm min}}{\lambda} \right)^2}  -1 .
\end{equation}
Equations (\ref{eq:x_of_s}) and (\ref{eq:y_of_s_rewritten}) describe the hanging cable parametrized by arc length.

\vskip 5pt

\section{Elastic cable or spider thread} \label{sec:elastic} Now we consider a cable that can stretch.
The threads of a spider web are an example. By passing to the limit of 
zero compliance, this discussion will also yield an alternative 
derivation of the standard hyperbolic cosine formula discussed in the 
preceding sections. This derivation is  less straightforward than the standard one explained earlier;
however, it explains the tangential tension forces more clearly.

\subsection{Background on Hooke's constant, compliance, and springs in series.} 

A linear spring with resting length $h$, extended to length $\ell$, contracts with force 
$
F = \kappa (\ell - h), 
$
where the constant of proportionality $\kappa$ is called {\em Hooke's constant}. 
Its reciprocal $c = 1/\kappa$ is called the {\em compliance} of the spring, so 
$
F = \frac{\ell-h}{c}.
$
The physical dimension of  $c$ is length per force. The greater the compliance, the easier is it to extend the spring.

\begin{figure}[h]
\begin{center}
\includegraphics[scale=0.35]{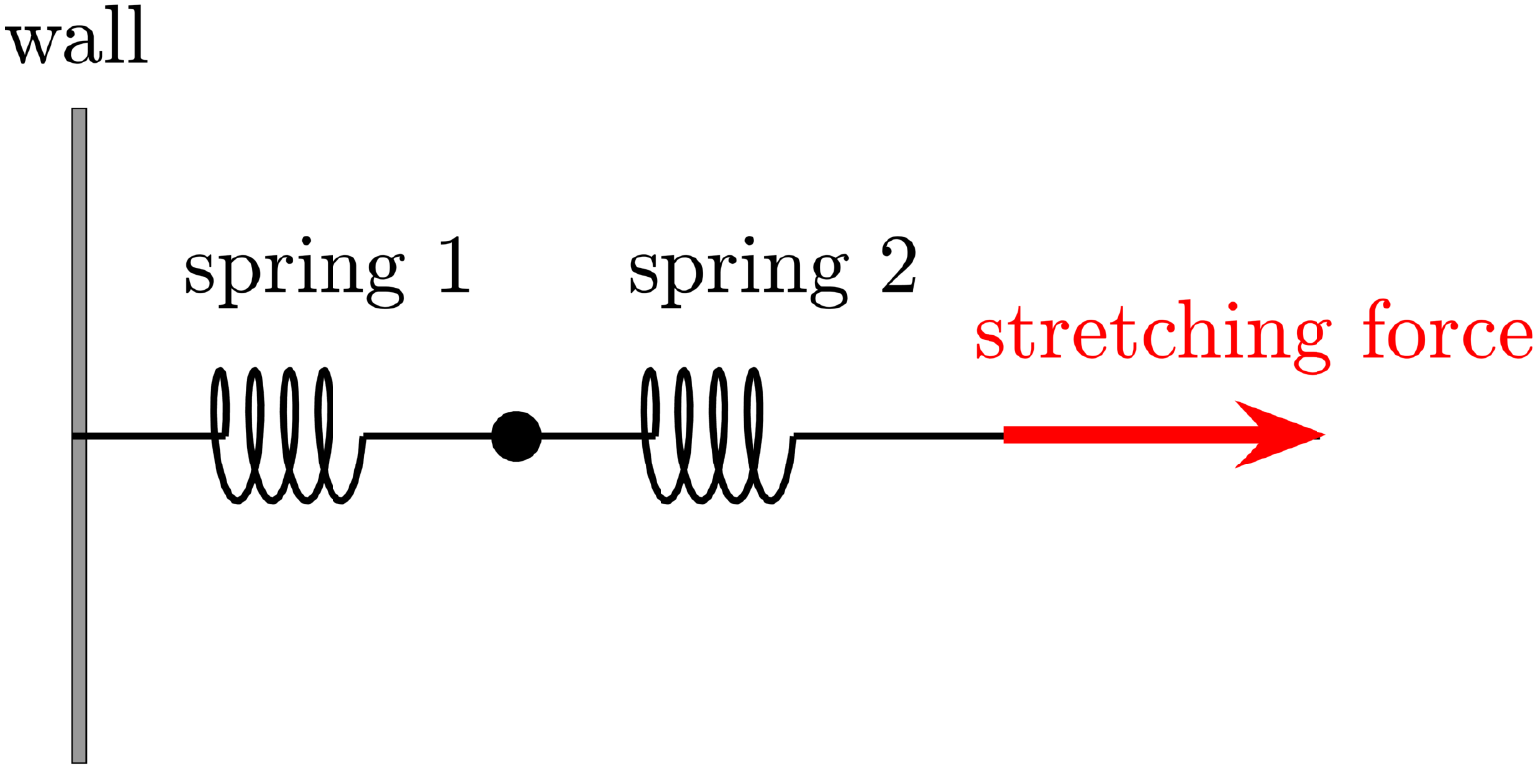}
\caption{Two springs in series being stretched.}
\label{fig:TWO_SPRINGS}
\end{center}
\end{figure}

Consider  now two springs in series, with  compliances $c_1$ and $c_2$ and resting lengths $h_1$ and $h_2$, 
attached on one end to a wall
as in Fig.\ \ref{fig:TWO_SPRINGS}.
Suppose you extend the springs from their combined resting length $h=h_1+h_2$ to some length $\ell$. The springs' lengths
will be $\ell_1$ and $\ell_2$, and Newton's third law implies that the springs pull on each other with precisely the 
overall stretching force:
\begin{equation}
\label{compliances}
\frac{\ell_1-h_1}{c_1} =   \frac{\ell_2-h_2}{c_2} = \frac{\ell-h}{c}
\end{equation}
where $c$ is the compliance of the combined spring made up of springs 1 and 2. From (\ref{compliances}), 
$$
\ell_1-h_1 =  \frac{c_1}{c} (\ell-h) ~~~~\mbox{and} ~~~ \ell_2-h_2 =  \frac{c_2}{c} (\ell-h).
$$
Summing these two equations, we find 
$$
\ell - h = \frac{c_1+c_2}{c} (\ell-h)
$$
and therefore 
\begin{equation}
c=c_1+c_2.
\end{equation}
The conclusion is that compliances add when springs are connected in series.

\subsection{String of mass points connected by springs.} 
\label{subsec:vertical}

Think about a string of finitely many mass points connected by massless springs. Later we will pass to a continuum 
limit.
I'll use the word ``cable" after passing to the continuum limit, but ``string" for the finitely many mass points
connected by springs.

So consider a string of $N+1$ mass points, connected by $N$ identical massless springs.  Assume that the resting lengths of the springs
are all the same; we denote them by $h$.  Assume that each spring has compliance $q h$, where $q>0$ is a fixed constant, called the {\em linear compliance density} (compliance per unit length),  a reciprocal force. Since compliances
sum when the springs are arranged in series, 
the compliance
of the string becomes $q h N = q L$, where $L$ is the length of the string when it is not 
under any tension.

Assume similarly that each mass point has mass 
$\rho h$, except for the two end points, which have mass $\rho h/2$, where $\rho>0$ is the linear mass
density. So altogether the mass of the string is $\rho h N =  \rho L$. Since $q$ is a reciprocal force,
the quantity 
\begin{equation}
\label{def_gamma}
\gamma = q  \cdot \rho L g = \mbox{linear compliance density} \cdot \mbox{weight of cable}
\end{equation}
is non-dimensional. It
quantifies the importance of elasticity for the cable, and will play an important role
in our analysis.

\begin{figure}[h]
\begin{center}
\includegraphics[scale=0.3]{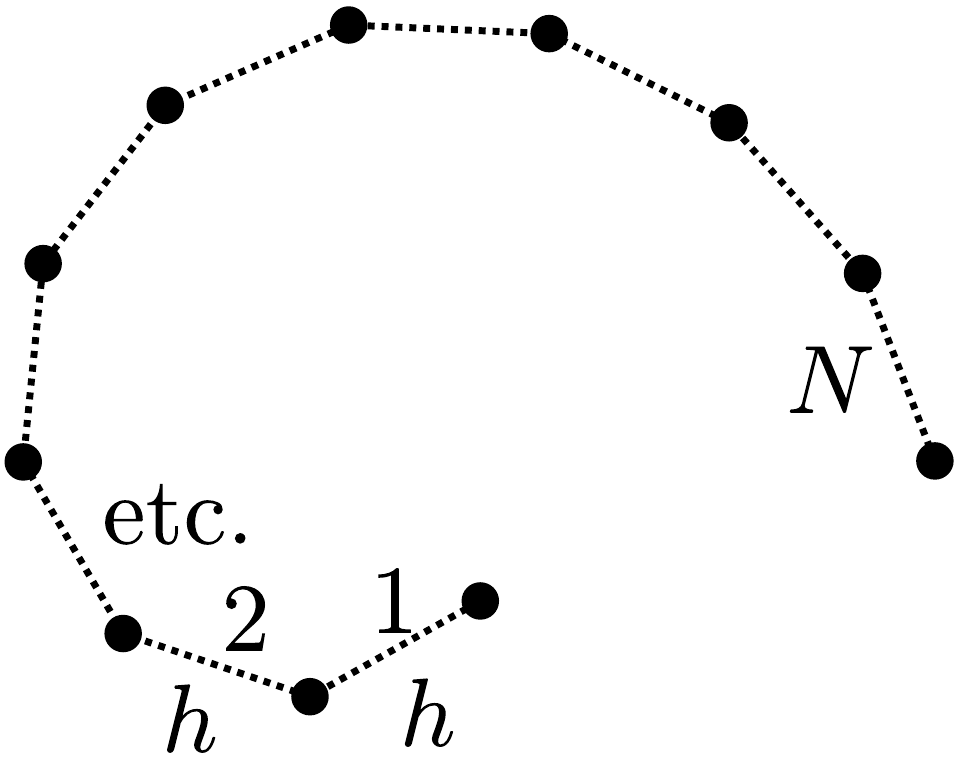}
\caption{The string of springs resting on the ground, under no tension.}
\end{center}
\end{figure}

\vskip 10pt
\subsection{String attached at both ends.}

Suppose now that we attach the string, as before, at $(x,y)=(A,H)$ and $(x,y)=(B,K)$, 
with $H \leq K$; see 
Fig.\ \ref{fig:CHAIN_ANGLES}. The position of the $i$-th mass point  is $(x_i,y_i)$, with
$$
(x_0,y_0) = (A,H) ~~~~~\mbox{and} ~~~~~ (x_N,y_N) = (B,K).
$$
We write 
$$
\ell_i = \sqrt{(x_i-x_{i-1})^2+(y_i-y_{i-1})^2}
$$
for the extended length of the $i$-th spring, 
and denote by $\alpha_i$ the angle between the $x$-axis and the $i$-th spring segment, $- \pi/2 < \alpha_i < \pi/2$; see Fig.\ \ref{fig:CHAIN_ANGLES}. 
We have 
\begin{equation}
\label{sincosalpha}
\sin \alpha_i =
\frac{y_i-y_{i-1}}{\ell_i}
~~~\mbox{and} ~~~
\cos \alpha_i =
\frac{x_i-x_{i-1}}{\ell_i}.
\end{equation}

\begin{figure}[h]
\begin{center}
\includegraphics[scale=0.3]{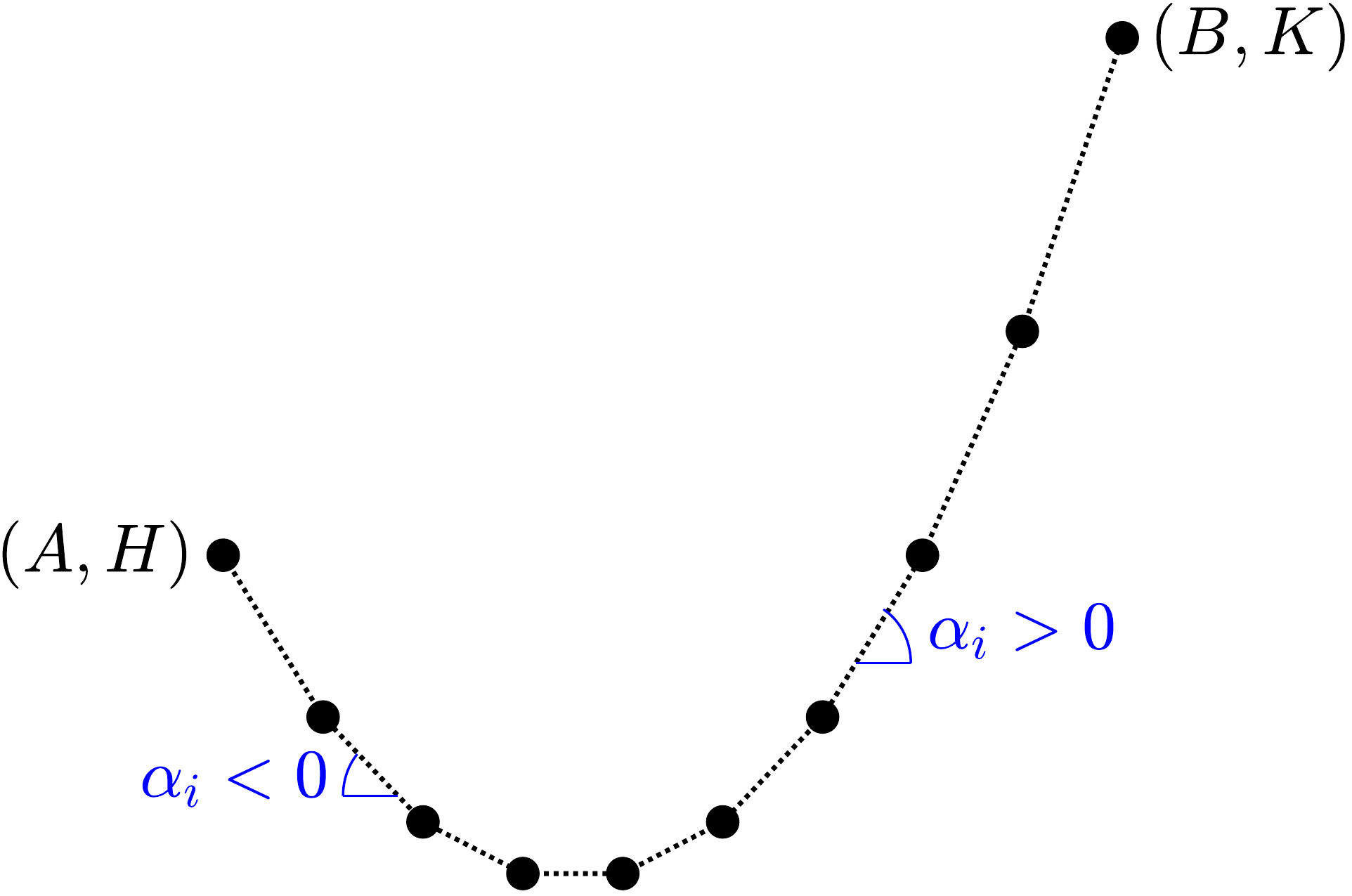}
\caption{The string of springs attached at both end points. The definition of the angles $\alpha_i$ (see text) is 
also indicated here.}
\label{fig:CHAIN_ANGLES}
\end{center}
\end{figure}

For $1 \leq i \leq N-1$, the total vertical force on mass point $i$ equals
$$
-\rho gh -  \frac{\ell_i - h}{q h} \sin \alpha_i + \frac{\ell_{i+1} - h}{q h} \sin \alpha_{i+1}. 
$$
This expression has to be zero, so we arrive at $N-1$ equations that must be satisfied when the cable hangs at rest: 
\begin{equation}
\label{eq:vertical_force_balance}
-\rho gh -  \frac{\ell_i - h}{q h} \sin \alpha_i + \frac{\ell_{i+1} - h}{q h} \sin \alpha_{i+1} = 0, ~~~ 1 \leq i \leq N-1.
\end{equation}
We will now transform these equations in such a way 
that difference quotients approximating derivatives appear, since we are planning to let $h \rightarrow 0$ so that a differential
equation emerges.

Using (\ref{sincosalpha}), we  re-write  (\ref{eq:vertical_force_balance}) as
$$
-\rho hg - \frac{1}{q h} \frac{\ell_i - h}{\ell_i} (y_i - y_{i-1}) + \frac{1}{q h} \frac{\ell_{i+1}- h}{\ell_{i+1}} (y_{i+1} - y_i) = 0.
$$
Multiplying both sides by $\frac{q}{h}$, and with a little bit of algebra:
$$
\frac{1}{h} \left( 
\frac{y_{i+1}- y_i}{h}  \left( 1 -  \frac{h}{\ell_{i+1}} \right)  - 
\frac{y_{i}- y_{i-1}}{h}  \left( 1 -  \frac{h}{\ell_{i}} \right)   \right) =  \rho q g, 
$$
so 
$$ 
 \frac{1}{h}
\left( 
\frac{y_{i+1}- y_i}{h}  \left( 1 -  \frac{1}{\sqrt{ \left( \frac{x_{i+1} - x_{i}}{h} \right)^2
+\left( \frac{y_{i+1} - y_i}{h} \right)^2}}
 \right)  -  \hskip 100pt \right.
 $$
 \begin{equation}
\label{eq:verti_equation_2}
\left. 
\frac{y_i- y_{i-1}}{h}  \left( 1 - \frac{1}{\sqrt{ \left( \frac{x_{i} - x_{i-1}}{h} \right)^2
+\left( \frac{y_i - y_{i-1}}{h} \right)^2  }}
 \right)  \right) =\rho q g.
\end{equation}
The balance of horizontal forces is expressed by the analogous equation 
$$
\frac{1}{h}
\left( 
\frac{x_{i+1}- x_{i}}{h}  \left( 1 -  \frac{1}{\sqrt{ \left( \frac{x_{i+1} - x_{i}}{h} \right)^2
+\left( \frac{y_{i+1} - y_i}{h} \right)^2}}
 \right)  -  \hskip 100pt \right.
 $$
 \begin{equation}
\label{eq:hori_equation_2}
\left. 
\frac{x_{i}- x_{i-1}}{h}  \left( 1 - \frac{1}{\sqrt{ \left( \frac{x_{i} - x_{i-1}}{h} \right)^2
+\left( \frac{y_i - y_{i-1}}{h} \right)^2  }}
 \right)  \right) = 0.
\end{equation}
The right-hand side is zero here because there is no horizontal gravitational force.

\subsection{Continuum limit.} 

We use arc length in the rest state, under no tension, as the independent variable, and denote it by $s$. At the left end, $s=0$, and at the right end, $s=L$. Equation  (\ref{eq:verti_equation_2}) 
is a finite difference discretization of 
$$
\frac{d}{ds} \left( \frac{dy}{ds} \left( 1 - \frac{1}{ \sqrt{ \left( \frac{dx}{ds} \right)^2 + \left( \frac{dy}{ds} \right)^2}} \right) \right) = \rho q g.
$$
The right-hand side of this equation equals $\gamma/L$ (see eq.\ (\ref{def_gamma})).
Integrating once, 
 \begin{equation}
 \label{eq:verti_equation_final}
\frac{dy}{ds} \left( 1 - \frac{1}{ \sqrt{ \left( \frac{dx}{ds} \right)^2 + \left( \frac{dy}{ds} \right)^2}} \right)  =  \frac{\gamma}{L} ~\!  \left( s -  s_{\rm min} \right)
 \end{equation}
 where $s_{\rm min}$ is the parameter corresponding to  the lowest point, at which $\frac{dy}{ds}=0$.
 Similarly, eq.\ (\ref{eq:hori_equation_2}) is a discretization of
$$
 \frac{d}{ds} \left( \frac{dx}{ds} \left( 1 - \frac{1}{ \sqrt{ \left( \frac{dx}{ds} \right)^2 + \left( \frac{dy}{ds} \right)^2}} \right) \right) = 0.
$$
Integrating once:
\begin{equation}
\label{eq:hori_equation_final}
\frac{dx}{ds}  \left( 1 - \frac{1}{ \sqrt{ \left( \frac{dx}{ds} \right)^2 + \left( \frac{dy}{ds} \right)^2}} \right)  = \mu
\end{equation}
for some non-dimensional constant $\mu$ yet to be discussed.

\subsection{Solving the differential equations.}

We simplify  eqs.\ (\ref{eq:verti_equation_final}) and (\ref{eq:hori_equation_final}) by solving for $dx/ds$ and $dy/ds$. 
First, write 
$$
R = \sqrt{ \left( \frac{dx}{ds} \right)^2 + \left( \frac{dy}{ds} \right)^2}.
$$
You may now say ``Wait, since $s$ is arc length, $ds^2 = dx^2+dy^2$, and therefore wouldn't $R$ always be equal to 1?" However, you have to remember that $s$ is arclength of the {\em unstretched} cable, before
it is hung up. Now, however, we are thinking of the cable as it hangs, and it is stretched; therefore $R>1$. 

With this notation,  (\ref{eq:verti_equation_final}) and (\ref{eq:hori_equation_final}) become
\begin{equation}
\label{eq:all_equations_final}
\frac{dx}{ds} = \frac{\mu}{1-\frac{1}{R}} ~~~\mbox{and} ~~~ \frac{dy}{ds} = \frac{\frac{\gamma}{L}  \left( s - s_{\rm min} \right)}{1-\frac{1}{R}}.
\end{equation}
Therefore 
$$
R =  \sqrt{ \left( \frac{dx}{ds} \right)^2 + \left( \frac{dy}{ds} \right)^2} = \sqrt{ \frac{\mu^2}{ \left( 1 - \frac{1}{R} \right)^2} + \frac{
\frac{ \gamma^2}{L^2}  \left( s - s_{\rm min} \right)^2} {\left( 1 - \frac{1}{R} \right)^2}}.
$$
Multiplying by $1-\frac{1}{R}$, we find:
$$
R-1 = \sqrt{\mu^2 +  \frac{ \gamma^2}{L^2}  \left( s -  s_{\rm min} \right)^2}
$$
and therefore 
$$
\frac{1}{1- \frac{1}{R}} = \frac{R}{R-1} = 1 + \frac{1}{R-1} = 1 + \frac{1}{ \sqrt{\mu^2 +  \frac{ \gamma^2}{L^2}  \left( s - s_{\rm min} \right)^2}}.
$$

Using this in (\ref{eq:all_equations_final}), we obtain
\begin{equation}
\label{eq:dxds} 
\frac{dx}{ds} =   \mu  + \frac{1}{\sqrt{1 +  \frac{ \gamma^2}{\mu^2 L^2}  \left( s - s_{\rm min} \right)^2}} 
\end{equation}
and 
\begin{equation}
\label{eq:dyds} 
\frac{dy}{ds}  =  \left( \mu + \frac{1}{\sqrt{1 +    \frac{\gamma^2}{\mu^2 L^2}  \left( s -  s_{\rm min} \right)^2}} \right) \frac{ \gamma}{\mu L}  \left( s - s_{\rm min} \right).
\end{equation} 

The combination $\frac{ \gamma}{\mu L}$ and its square appear three times in eqs.\ (\ref{eq:dxds}) and 
(\ref{eq:dyds}). We simplify the notation by defining
$$
\lambda = \frac{\mu L}{\gamma}. 
$$
Since $\mu$ and $\gamma$ are non-dimensional, $\lambda$ is a length.
It will turn out to be the natural analogue in the elastic case of the shape parameter.
(This isn't obvious at this point, or at least it wasn't to me. I realized it only after having done the calculation that's about to
follow.)
With this notation, (\ref{eq:dxds}) and (\ref{eq:dyds}) become 
$$ 
\frac{dx}{ds} =  \gamma \frac{\lambda}{L} + \frac{1}{\sqrt{1 + \left( \frac{  s - s_{\rm min}}{\lambda} \right)^2}}
~~~\mbox{and} ~~~ 
\frac{dy}{ds} =  \left(  \gamma \frac{\lambda}{L} + \frac{1}{\sqrt{1 + \left( \frac{s - s_{\rm min} }{\lambda} \right)^2}} \right) \frac{s-s_{\rm min} }{\lambda}.
$$
We integrate again to obtain formulas for $x$ and $y$:
\begin{equation}
\label{eq:x_param}
\frac{x-x_{\rm min}}{\lambda}  =     ~ \sinh^{-1} \left(  \frac{s- s_{\rm min} }{\lambda} \right)  \!+~\!  \gamma \frac{\lambda}{L}  \frac{s-s_{\min}}{\lambda}, 
\end{equation}
and 
\begin{equation}
\label{eq:y_param}
\frac{y-y_{\rm min}}{\lambda} =  \sqrt{1 + \left( \frac{s-s_{\rm min}}{\lambda} \right)^2} - 1  ~\!+~\! \gamma  \frac{\lambda}{L}    ~\! \frac{1}{2} \left( \frac{s - s_{\rm min}}{\lambda} \right)^2 
\end{equation}
where $x_{\rm min}$ and $y_{\rm min}$ are the values of $x$ and $y$ when $s=s_{\rm min}$. The constants of integration were chosen to make sure that $s=s_{\rm min}$ corresponds to $x = x_{\rm min}$ and $y=y_{\rm min}$. 

Equations (\ref{eq:x_param}) and (\ref{eq:y_param}) are very similar to eqs.\ (\ref{eq:x_of_s}) and (\ref{eq:y_of_s_rewritten}), the arclength parametrization of
the inelastic catenary. The only difference here are the extra summands
proportional to $\gamma$. 
These terms disappear as the compliance  density $q$ tends to zero (recall $\gamma =  q \rho L g$), so in the limit of vanishing 
compliance density, we obtain the standard description of the inelastic catenary, which we have
thereby re-derived. 

Although this derivation is more involved than the standard one, I prefer it because it paints a microscopic picture of 
the ``tensions forces"  --- even if that microscopic picture is, of course, an idealization.

\subsection{Equations for the parameters.}
The four parameters $\lambda$, $s_{\rm min}$, $x_{\rm min}$, and $y_{\rm min}$ must be chosen so that the conditions
$$
x(0)=A, ~~~y(0)=H, ~~~ x(L)=B, ~~~ y(L)=K
$$
are satisfied.
Using eqs.\ (\ref{eq:x_param}) and (\ref{eq:y_param}), $x(0)=A$ and $y(0)=H$ mean
\begin{equation}
\label{eq:x_min_is}
x_{\rm min} =   \gamma ~\! \frac{ \lambda}{L}  ~\! s_{\rm min} ~+~  A + \lambda \sinh^{-1} \left( \frac{s_{\rm min}}{\lambda} \right)
\end{equation}
and
\begin{equation}
\label{eq:x_min_is}
y_{\rm min} =  - \gamma  ~\! \left( \frac{\lambda}{L} \right)^2 ~\! \frac{L}{2} ~\!  \left(  \frac{s_{\rm min}}{\lambda} \right)^2  ~+~ 
H - \lambda \left(  \sqrt{1 + \left( \frac{s_{\min}}{\lambda} \right)^2} - 1 \right).
\end{equation}

Given that $x(0)=A$ and $y(0)=H$, the remaining two conditions can equivalently be written as $x(L)-x(0)=B-A$, and $y(L)-y(0)=K-H$. Using 
eqs.\ (\ref{eq:x_param}) and (\ref{eq:y_param}), these equations become
\begin{eqnarray} 
\label{bc3_explicit}
 \gamma ~\!+~\! \sinh^{-1} \left(  \frac{L- s_{\rm min}}{\lambda} \right)  +  \sinh^{-1}  \left( \frac{s_{\rm min}}{\lambda}  \right) &=&  \frac{B-A}{\lambda}, \\
\nonumber
 \gamma ~\! \frac{\lambda}{L} ~\!  \frac{1}{2} ~\! \left(   \left(  \frac{L- s_{\rm min}}{\lambda} \right)^2 - \left( \frac{s_{\rm min}}{\lambda} \right)^2  \right)  ~\!+~\!  \hskip 50pt  \\
\label{bc4_explicit}
\hskip 35pt
 \sqrt{1+ \left( \frac{L-s_{\rm min}}{\lambda} \right)^2} - \sqrt{1+ \left( \frac{s_{\rm min}}{\lambda} \right)^2}  &=& \frac{K-H}{\lambda}. 
\end{eqnarray}

Unfortunately,  the extra terms proportional to $\gamma$ in eqs.\ (\ref{bc3_explicit}) and (\ref{bc4_explicit}) undermine the algebra that uncoupled 
the equations earlier, at least as far as I can see. We must solve
 eqs.\ (\ref{bc3_explicit}) and (\ref{bc4_explicit})  jointly 
for $(\lambda, s_{\rm min})$ now.

\subsection{Numerical computation of shape parameter and lowest point.} 
When discussing the inelastic case, I found it convenient to solve for
$$
\xi = \frac{B-A}{2 \lambda}
$$
instead of directly for $\lambda$. I did the same thing here, replacing $\lambda$ by $\frac{B-A}{2 \xi}$ in 
(\ref{bc3_explicit}) and (\ref{bc4_explicit}). Then I solved the resulting equations 
using two-dimensional Newton iteration. Figure \ref{fig:ELASTIC_CHAIN_SIMPLE_FORMULAS} shows examples, with $\gamma$ rising from $0$ to $2$ in steps of 0.2. For each new value of $\gamma$, 
I used the parameters $\xi$ and $s_{\rm min}$ computed for the preceding value as starting values for the 
Newton iteration. The parameters for the elastic catenaries in Fig.\ \ref{fig:ELASTIC_CHAIN_SIMPLE_FORMULAS} are computed in six or fewer Newton 
iterations with 15-digit accuracy. The parameter $\xi$ in the inelastic case takes a bit longer, 9 Newton iterations. 
\begin{figure}[h]
\begin{center}
\includegraphics[scale=0.35]{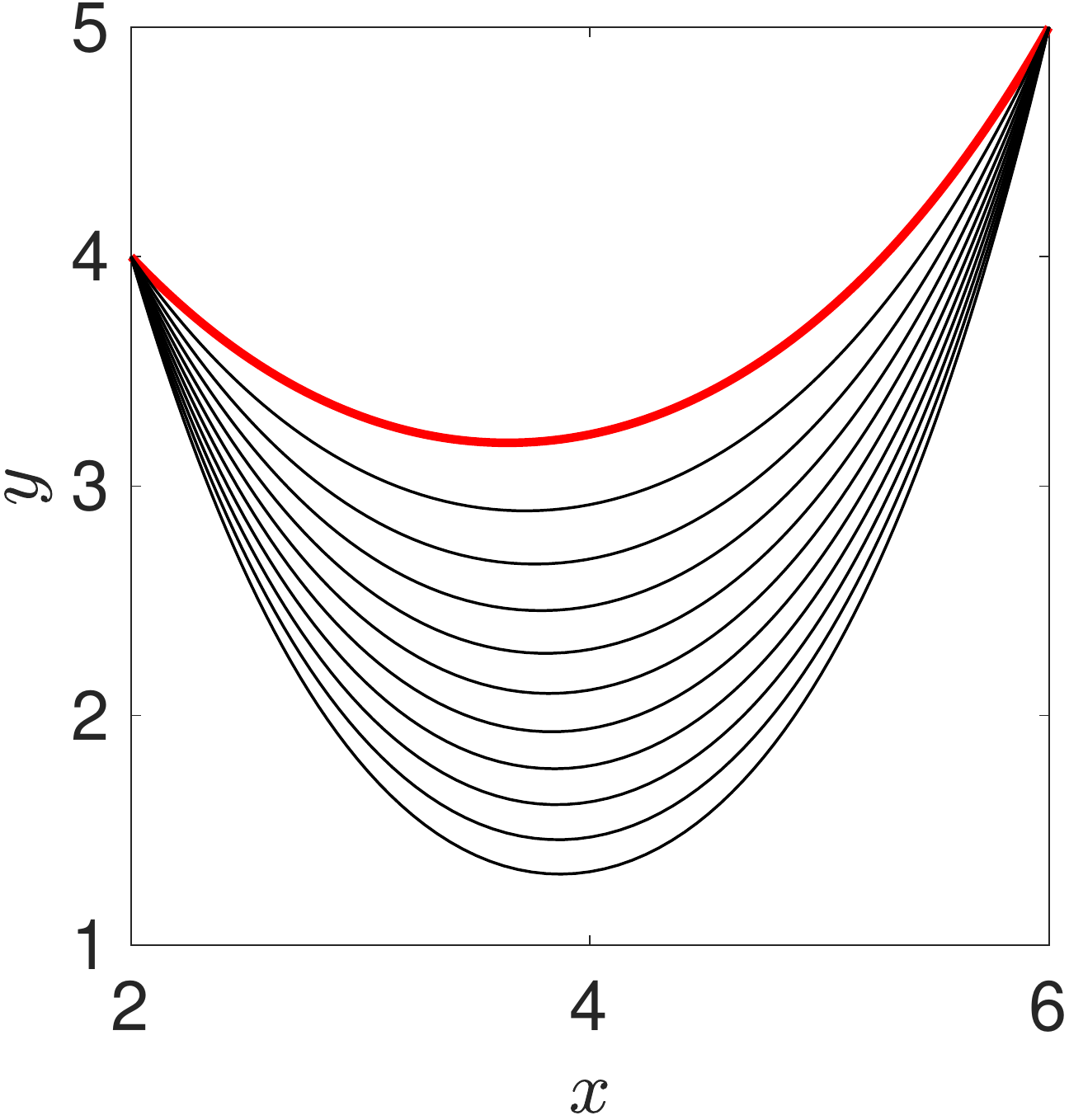}
\caption{An inelastic cable (red), and increasingly elastic cables (black), with $\gamma$ varying from $0$ to $2$ in 
steps of $0.2$.}
\label{fig:ELASTIC_CHAIN_SIMPLE_FORMULAS}
\end{center}
\end{figure}

It does not appear necessary to use this ``continuation" approach. I have not encountered a single example
in which Newton's method,  starting with the values
$\xi$ and $s_{\rm min}$ computed for the inelastic case, did not converge rapidly, even when $\gamma$ is taken
to be very large.

\section{Concluding comments} 

\subsection{Code for Figures 6 and 10.} I will be happy to send you the 
Matlab code generating Figures \ref{fig:TWO_EXAMPLES_NEW} and \ref{fig:ELASTIC_CHAIN_SIMPLE_FORMULAS}; just e-mail me. 

\subsection{Loose ends.} 
Several questions remain unanswered here. First, can we prove that eqs.\ (\ref{bc3_explicit}) and (\ref{bc4_explicit}) 
have a unique solution $(\lambda,s_{\rm min})$ with $\lambda>0$ and $s_{\rm min} \leq \frac{L}{2}$, for any 
choice of $A<B$, $H \leq K$, and $L>0$? That
ought to be the case, but I haven't proved it. 

Second, why does Newton's method for  (\ref{bc3_explicit}) and (\ref{bc4_explicit}), starting with the parameters  for the inelastic cable, always seem to work so well, even when
the compliance is large? Reassuringly, if there were a case in which it didn't converge rapidly, one {\em could} always use 
continuation, raising the compliance gradually, and that would certainly work. However, in my experience it never
seems necessary. 

What if the springs were not linear? In that case, analogues of the 
differential  equations  (\ref{eq:verti_equation_final}) and  (\ref{eq:hori_equation_final}) can still be written down, but they
cannot in general be solved explicitly for $x$ and $y$, so there are no analogues of (\ref{eq:x_param}) and (\ref{eq:y_param}) any
more.

\subsection{Engingineering literature on hanging cables.} 

None of what I have presented here could conceivably be new. In fact there is an extensive 
sophisticated
engineering-oriented literature of which the catenary problem is merely the starting point; see \cite{Greco_et_al,Huang_et_al,Pierce_et_al_1913,Catenary_Analysis,Tang_et_al_2021} for a few examples. However, 
what I have presented here seems difficult if not impossible to extract from that literature.

\subsection{Spider webs.}
One could  consider multiple elastic cables attached to each other, as in a spider web. The ``discrete" model, thinking of spider threads as composed of mass points connected by massless springs, is  straightforward to formulate. There is a substantial literature on the mathematical and computational modeling of spider webs; see
\cite{Aoyanagi_2010,Kawano_2019,Lin_et_al_1995} for a few examples.

\begin{acknowledgment} This article was inspired by Mark Levi's  beautiful discussion of some of the astonishing properties of
catenaries in the May 2021 issue of  {\em SIAM News} \cite{Levi}. I would like to thank the anonymous reviewer for reading my paper so thoughtfully, correcting typos and suggesting improvements. 
\end{acknowledgment}

\begin{biog}
\item[Christoph B\"orgers]  is a Professor of Mathematics at Tufts University. His research interests include mathematical neuroscience,  numerical analysis, and more recently anomalous diffusion
and opinion dynamics.  In
2022 he was the recipient of Tufts University's Leibner Award for Excellence in Teaching and Advising.
\end{biog}

\begin{thebibliography}{10}

\bibitem{Aoyanagi_2010} {\sc Y. Aoyanagi and K. Okumura} (2010). Simple model for the 
mechanics of spider webs. {\em Phys.\ Rev.\ Lett.} 104: 038102.
doi.org/10.1103/PhysRevLett.104.038102

\bibitem{Bernoulli_catenary}
{\sc J.~Bernoulli} (1691). Solutio problematis funicularii. {\em Acta Erud.}:
274--276.

\bibitem{Boehme_et_al_1980}
{\sc R.~B\"ohme, S.~Hildebrandt, and E.~Tausch} (1980). The two-dimensional analogue of the catenary. {\em Pac.\ J.\ Math.}\ 88(2): 247--278.
doi.org/10.2140/PJM.1980.88.247

\bibitem{de_la_Grandville_2022}
{\sc O.~de~La~Grandville} (2022). On a Classic Problem in the Calculus of Variations: Setting Straight Key Properties of the Catenary. {\em Am.\ Math.\ Mon.} 129(2): 103--115. doi.org/10.1080/00029890.2022.2004849



\bibitem{Elastic_catenary}
{\sc E.~Bobillier and P.J.E.~Finck} (1826-1827).  Solution des deux probl\`emes de
  statique propos\'es \`a la page 296 du pr\'ec\'edent volume. {\em Ann. Math. Pures Appl.}\ 17:  59--68.
eudml.org/doc/80155


\bibitem{Gohnert}
{\sc M.~Gohnert} (2022). {\em Shell Structures, Theory and Application}. Springer International Publishing.

  
  \bibitem{Greco_et_al}
{\sc L.~Greco, N.~Impollonia, and M.~Cuomo} (2014). A procedure for the static
  analysis of cable structures following elastic catenary theory.
  {\em Int.\ J.\ Solids Struct.}\ 51(7--8): 1521--1533.  \\ doi.org/10.1016/j.ijsolstr.2014.01.001
  

  \bibitem{Huang_et_al}
{\sc W.~Huang, D.~He, D.~Tong, Y.~Chen, X.~Huang, L.~Qin, and Q.~Fei} (2022). 
  Static analysis of elastic cable structures under mechanical load using
  discrete catenary theory, {\em Fundam.\ Res.}: in press. \\
doi.org/10.1016/j.fmre.2022.03.011
  
  \bibitem{Huygens_catenary}
{\sc C.~Huygens} (1691). Dynaste {Z}ulchemii, solutio problematis funicularii.
  {\em Acta Erud.}: 281--282.
  
  \bibitem{Kawano_2019}
  {\sc A. Kawano and A. Morassi} (2019). Detecting a pray in a spider orb web. {\em SIAM J.\ Appl.\ Math.}\ 79(6): 2506--2529.
doi.org/10.1137/20M1372792
  
  
\bibitem{Leibniz_catenary}
{\sc G.~Leibniz} (1691).  Solutio problematis catenarii. {\em Acta Erud.}:
277--281.

  \bibitem{Levi} 
  {\sc M.~Levi} (2021). Hanging cables and hydrostatics. {\em SIAM News} 54, No. 4.   \\
sinews.siam.org/Details-Page/hanging-cables-and-hydrostatics
  
  \bibitem{Lin_et_al_1995}
  {\sc L.H.~Lin, D.T.~Edmonds, and F.~Vollrath} (1995).  Stuctural engineering of an orb-spider's web. {\em Nature} 373: 146--148.
  doi.org/10.1038/373146a0 
  
  \bibitem{Lopez_2022} 
  {\sc R.~L\'opez} (2022). A dome subjected to compression forces: {A} comparison study between the mathematical
  model, the catenary rotation surface and the paraboloid. {\em Chaos Solit. Fractals} 161: 112350. \\ doi.org/10.1016/j.chaos.2022.112350

\bibitem{Osserman}
{\sc R.~Osserman} (2010). How the {Gateway Arch} got its shape.  {\em Nexus Netw. J.} 12(2): 167--189. \\
doi.org/10.1007/s00004-010-0030-8

\bibitem{Pierce_et_al_1913}
{\sc C.~A. Pierce, F.~J. Adams, and G.~I. Gilchrest} (1913). Theory of the
  non-elastic and elastic catenary as applied to transmission lines, 
  {\em Proc. Inst. Electr. Eng.}\ 32(6):  1373--1391. \\
  doi.org/10.1109/PAIEE.1913.6660750

\bibitem{Catenary_Analysis}
{\sc J.~Qin, J.~Chen, L.~Qiao, J.~Wan, and Y.~Xia} (2016).  Catenary analysis and
  calculation method of track rope of cargo cableway with multiple loads.
  {\em MATEC Web of Conferences} 82: 01008. \\ doi.org/10.1051/matecconf/20168201008
  
  \bibitem{Tang_et_al_2021}
{\sc H.-B.~Tang, Y.~Han, H.~Fu, and B.G.~Xu} (2021). Mathematical modeling of linearly-elastic 
non-prestrained cables based on a local reference frame. 
{\em Appl.\ Math.\ Model.} 91: 695--708.  \\ doi.org/10.1016/j.apm.2020.10.008


\end{thebibliography}
\end{document}